\documentclass{article}
\usepackage{amsmath,amsfonts,amssymb,amsthm,mathtools,bbm}
\usepackage{microtype}
\usepackage{lipsum}
\usepackage{url,hyperref}
\usepackage{indentfirst}
\hypersetup{colorlinks, linkcolor={red}, citecolor={blue}, urlcolor={black}}
\usepackage{comment}
\usepackage[margin=1in]{geometry}
\usepackage{enumerate,enumitem}
\usepackage{subcaption}
\usepackage{tikz}
\usetikzlibrary{calc,matrix,positioning,angles,quotes}
\usetikzlibrary{shapes.geometric}
\usetikzlibrary{patterns,decorations.pathreplacing}
\usepackage{array,multirow}

\allowdisplaybreaks

\DeclareMathOperator\per{per}
\DeclareMathOperator\sgn{sgn}

\newtheorem{theorem}{Theorem}[section]
\newtheorem{lemma}[theorem]{Lemma}

\newtheorem{corollary}[theorem]{Corollary}

\newtheorem{claim}[theorem]{Claim}

\title{Exact enumeration of lozenge tilings of a triangular region}
\author{Jun Yan\thanks{Email: \url{j228yan@uwaterloo.ca}. Part of the work was carried out while supported by
ERC Advanced Grant 883810.}}
\date{}

\begin{document}

\maketitle
\begin{abstract}
We prove that the number of lozenge tilings of a certain triangular region $\mathcal{T}_n$ is given by the formula
\[T_n=\prod_{\substack{1\leq a<b\leq 3n+2\\(a,b)\not=(n+1,2n+2)}}\left|1+\zeta^a+\zeta^b\right|^{1/3},\]
where $\zeta=e^{2\pi i/(3n+3)}$.
This answers a question of Ciucu and Krattenthaler, both by finding the exact formula and by explaining why $T_n$ has many prime factors. The proof reduces the lozenge tiling enumeration problem to evaluating the determinant of the bipartite adjacency matrix $M_n$ of the dual graph of $\mathcal{T}_n$, and then evaluates this determinant by diagonalising $M_n$.
\end{abstract}

\section{Introduction}\label{sec:intro}
Given a geometric region $\mathcal{R}$, a tiling is a covering of $\mathcal{R}$ using some given smaller pieces, so that there is no gap or overlap. Tiling enumeration is an active subfield of Enumerative Combinatorics. It has many fascinating connections with the enumeration of other combinatorial objects like perfect matchings, plane partitions, and lattice paths, and with other areas of mathematics including representation theory, linear algebra, and mathematical physics. We refer interested readers to~\cite{Propp} for a comprehensive survey by Propp on tiling enumeration.

A classical tiling enumeration result was obtained independently in 1961 by Kasteleyn~\cite{K}, and by Temperley and Fisher~\cite{TF}. They proved that the number of tilings of the $m\times n$ grid of squares using dominoes (see Figure~\ref{fig:domloz}), which are $1\times2$ or $2\times1$ rectangles, is given by the beautiful formula
\[\left(\prod_{i=1}^m\prod_{j=1}^n\left(4\cos^2\frac{i\pi}{m+1}+4\cos^2\frac{j\pi}{n+1}\right)\right)^{1/4}.\]

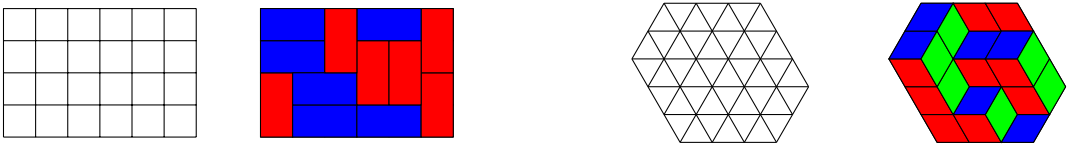
\begin{figure}[h]
    \centering
    \begin{minipage}{0.49\textwidth}
\begin{center}
\begin{tikzpicture}[scale=0.85]
\draw[step=0.5cm,color=black] (-4,-3) grid (-1,-1);

\draw[step=0.5cm,color=black] (0,-3) grid (3,-1);
\filldraw [fill=red, draw=black] (0,-3) rectangle (0.5,-2);
\filldraw [fill=blue, draw=black] (0,-2) rectangle (1,-1.5);
\filldraw [fill=blue, draw=black] (0,-1.5) rectangle (1,-1);
\filldraw [fill=blue, draw=black] (0.5,-3) rectangle (1.5,-2.5);
\filldraw [fill=blue, draw=black] (0.5,-2.5) rectangle (1.5,-2);
\filldraw [fill=blue, draw=black] (1.5,-3) rectangle (2.5,-2.5);
\filldraw [fill=red, draw=black] (2.5,-3) rectangle (3,-2);
\filldraw [fill=red, draw=black] (2.5,-2) rectangle (3,-1);
\filldraw [fill=red, draw=black] (1.5,-2.5) rectangle (2,-1.5);
\filldraw [fill=red, draw=black] (2,-2.5) rectangle (2.5,-1.5);
\filldraw [fill=blue, draw=black] (1.5,-1.5) rectangle (2.5,-1);
\filldraw [fill=red, draw=black] (1,-2) rectangle (1.5,-1);
\end{tikzpicture}

\end{center}
\end{minipage}
\begin{minipage}{0.49\textwidth}
\begin{center}
\begin{tikzpicture}[scale=0.85]
\draw (0,0) -- (1.5,0);
\draw ($(-0.25,{0.25*sqrt(3)})$) -- ($(1.75,{0.25*sqrt(3)})$);
\draw ($(-0.5,{0.5*sqrt(3)})$) -- ($(2,{0.5*sqrt(3)})$);
\draw ($(-0.75,{0.75*sqrt(3)})$) -- ($(1.75,{0.75*sqrt(3)})$);
\draw ($(-0.5,{sqrt(3)})$) -- ($(1.5,{sqrt(3)})$);
\draw ($(-0.25,{1.25*sqrt(3)})$) -- ($(1.25,{1.25*sqrt(3)})$);

\draw ($(-0.75,{0.75*sqrt(3)})$) -- ($(-0.25,{1.25*sqrt(3)})$);
\draw ($(-0.5,{0.5*sqrt(3)})$) -- ($(0.25,{1.25*sqrt(3)})$);
\draw ($(-0.25,{0.25*sqrt(3)})$) -- ($(0.75,{1.25*sqrt(3)})$);
\draw ($(0,{0*sqrt(3)})$) -- ($(1.25,{1.25*sqrt(3)})$);
\draw ($(0.5,{0*sqrt(3)})$) -- ($(1.5,{1*sqrt(3)})$);
\draw ($(1,{0*sqrt(3)})$) -- ($(1.75,{0.75*sqrt(3)})$);
\draw ($(1.5,{0*sqrt(3)})$) -- ($(2,{0.5*sqrt(3)})$);

\draw ($(2,{0.5*sqrt(3)})$) -- ($(1.25,{1.25*sqrt(3)})$);
\draw ($(1.75,{0.25*sqrt(3)})$) -- ($(0.75,{1.25*sqrt(3)})$);
\draw ($(1.5,{0*sqrt(3)})$) -- ($(0.25,{1.25*sqrt(3)})$);
\draw ($(1,{0*sqrt(3)})$) -- ($(-0.25,{1.25*sqrt(3)})$);
\draw ($(0.5,{0*sqrt(3)})$) -- ($(-0.5,{1*sqrt(3)})$);
\draw ($(0,{0*sqrt(3)})$) -- ($(-0.75,{0.75*sqrt(3)})$);

\draw (4,0) -- (5.5,0);
\draw ($(3.75,{0.25*sqrt(3)})$) -- ($(5.75,{0.25*sqrt(3)})$);
\draw ($(3.5,{0.5*sqrt(3)})$) -- ($(6,{0.5*sqrt(3)})$);
\draw ($(3.25,{0.75*sqrt(3)})$) -- ($(5.75,{0.75*sqrt(3)})$);
\draw ($(3.5,{sqrt(3)})$) -- ($(5.5,{sqrt(3)})$);
\draw ($(3.75,{1.25*sqrt(3)})$) -- ($(5.25,{1.25*sqrt(3)})$);

\draw ($(3.25,{0.75*sqrt(3)})$) -- ($(3.75,{1.25*sqrt(3)})$);
\draw ($(3.5,{0.5*sqrt(3)})$) -- ($(4.25,{1.25*sqrt(3)})$);
\draw ($(3.75,{0.25*sqrt(3)})$) -- ($(4.75,{1.25*sqrt(3)})$);
\draw ($(4,{0*sqrt(3)})$) -- ($(5.25,{1.25*sqrt(3)})$);
\draw ($(4.5,{0*sqrt(3)})$) -- ($(5.5,{1*sqrt(3)})$);
\draw ($(5,{0*sqrt(3)})$) -- ($(5.75,{0.75*sqrt(3)})$);
\draw ($(5.5,{0*sqrt(3)})$) -- ($(6,{0.5*sqrt(3)})$);

\draw ($(6,{0.5*sqrt(3)})$) -- ($(5.25,{1.25*sqrt(3)})$);
\draw ($(5.75,{0.25*sqrt(3)})$) -- ($(4.75,{1.25*sqrt(3)})$);
\draw ($(5.5,{0*sqrt(3)})$) -- ($(4.25,{1.25*sqrt(3)})$);
\draw ($(5,{0*sqrt(3)})$) -- ($(3.75,{1.25*sqrt(3)})$);
\draw ($(4.5,{0*sqrt(3)})$) -- ($(3.5,{1*sqrt(3)})$);
\draw ($(4,{0*sqrt(3)})$) -- ($(3.25,{0.75*sqrt(3)})$);

\filldraw [fill=red, draw=black] (4,0) -- (4.5,0) -- ($(4.25,{0.25*sqrt(3)})$) -- ($(3.75,{0.25*sqrt(3)})$) -- cycle;
\filldraw [fill=red, draw=black] ($(3.75,{0.25*sqrt(3)})$) -- ($(4.25,{0.25*sqrt(3)})$) -- ($(4,{0.5*sqrt(3)})$) -- ($(3.5,{0.5*sqrt(3)})$) -- cycle;
\filldraw [fill=red, draw=black] (4.5,0) -- (5,0) -- ($(4.75,{0.25*sqrt(3)})$) -- ($(4.25,{0.25*sqrt(3)})$) -- cycle;
\filldraw [fill=green, draw=black] (5,0) -- ($(5.25,{0.25*sqrt(3)})$) -- ($(5,{0.5*sqrt(3)})$) -- ($(4.75,{0.25*sqrt(3)})$) -- cycle;
\filldraw [fill=blue, draw=black] (5,0) -- ($(5.5,0)$) -- ($(5.75,{0.25*sqrt(3)})$) -- ($(5.25,{0.25*sqrt(3)})$) -- cycle;
\filldraw [fill=red, draw=black] ($(5,{0.5*sqrt(3)})$) -- ($(5.25,{0.25*sqrt(3)})$) -- ($(5.75,{0.25*sqrt(3)})$) -- ($(5.5,{0.5*sqrt(3)})$) -- cycle;
\filldraw [fill=red, draw=black] ($(5,{0.5*sqrt(3)})$) -- ($(4.5,{0.5*sqrt(3)})$) -- ($(4.25,{0.75*sqrt(3)})$) -- ($(4.75,{0.75*sqrt(3)})$) -- cycle;
\filldraw [fill=blue, draw=black] ($(5,{0.5*sqrt(3)})$) -- ($(4.5,{0.5*sqrt(3)})$) -- ($(4.25,{0.25*sqrt(3)})$) -- ($(4.75,{0.25*sqrt(3)})$) -- cycle;
\filldraw [fill=green, draw=black] ($(4.25,{0.25*sqrt(3)})$) -- ($(4.5,{0.5*sqrt(3)})$) -- ($(4.25,{0.75*sqrt(3)})$) -- ($(4,{0.5*sqrt(3)})$) -- cycle;
\filldraw [fill=red, draw=black] ($(3.75,{0.75*sqrt(3)})$) -- ($(3.25,{0.75*sqrt(3)})$) -- ($(3.5,{0.5*sqrt(3)})$) -- ($(4,{0.5*sqrt(3)})$) -- cycle;
\filldraw [fill=blue, draw=black] ($(3.75,{0.75*sqrt(3)})$) -- ($(3.25,{0.75*sqrt(3)})$) -- ($(3.5,{1*sqrt(3)})$) -- ($(4,{1*sqrt(3)})$) -- cycle;
\filldraw [fill=blue, draw=black] ($(4,{1*sqrt(3)})$) -- ($(3.5,{1*sqrt(3)})$) -- ($(3.75,{1.25*sqrt(3)})$) -- ($(4.25,{1.25*sqrt(3)})$) -- cycle;
\filldraw [fill=green, draw=black] ($(4,{1*sqrt(3)})$) -- ($(4.25,{0.75*sqrt(3)})$) -- ($(4,{0.5*sqrt(3)})$) -- ($(3.75,{0.75*sqrt(3)})$) -- cycle;
\filldraw [fill=green, draw=black] ($(4.25,{1.25*sqrt(3)})$) -- ($(4.5,{1*sqrt(3)})$) -- ($(4.25,{0.75*sqrt(3)})$) -- ($(4,{1*sqrt(3)})$) -- cycle;
\filldraw [fill=blue, draw=black] ($(4.25,{0.75*sqrt(3)})$) -- ($(4.5,{1*sqrt(3)})$) -- ($(5,{1*sqrt(3)})$) -- ($(4.75,{0.75*sqrt(3)})$) -- cycle;
\filldraw [fill=red, draw=black] ($(4.25,{1.25*sqrt(3)})$) -- ($(4.5,{1*sqrt(3)})$) -- ($(5,{1*sqrt(3)})$) -- ($(4.75,{1.25*sqrt(3)})$) -- cycle;
\filldraw [fill=red, draw=black] ($(4.75,{1.25*sqrt(3)})$) -- ($(5,{1*sqrt(3)})$) -- ($(5.5,{1*sqrt(3)})$) -- ($(5.25,{1.25*sqrt(3)})$) -- cycle;
\filldraw [fill=red, draw=black] ($(4.75,{0.75*sqrt(3)})$) -- ($(5,{0.5*sqrt(3)})$) -- ($(5.5,{0.5*sqrt(3)})$) -- ($(5.25,{0.75*sqrt(3)})$) -- cycle;
\filldraw [fill=blue, draw=black] ($(4.75,{0.75*sqrt(3)})$) -- ($(5,{1*sqrt(3)})$) -- ($(5.5,{1*sqrt(3)})$) -- ($(5.25,{0.75*sqrt(3)})$) -- cycle;
\filldraw [fill=green, draw=black] ($(5.5,{0.5*sqrt(3)})$) -- ($(5.25,{0.75*sqrt(3)})$) -- ($(5.5,{1*sqrt(3)})$) -- ($(5.75,{0.75*sqrt(3)})$) -- cycle;
\filldraw [fill=green, draw=black] ($(5.75,{0.25*sqrt(3)})$) -- ($(5.5,{0.5*sqrt(3)})$) -- ($(5.75,{0.75*sqrt(3)})$) -- ($(6,{0.5*sqrt(3)})$) -- cycle;
\end{tikzpicture}
\end{center}
\end{minipage}
    \caption{On the left, the $4\times 6$ grid of squares and a domino tiling. On the right, the centrally symmetric hexagonal region $H(3,3,2)$ and a lozenge tiling.}\label{fig:domloz}
\end{figure}

Another commonly studied type of tiling enumeration problem concerns lozenge tilings. In this setting, the region is a grid of equilateral triangles, and the small pieces used are lozenges, which are two equilateral triangles glued together along an edge. In 1989, David and Tomei~\cite{DT} observed that lozenge tilings of the centrally symmetric hexagonal region $H(a,b,c)$ with side lengths $a,b,c,a,b,c$ in order (see Figure~\ref{fig:domloz}) are in bijection with plane partitions that fit inside an $a\times b\times c$ box. A celebrated result of MacMahon~\cite{M} proves that the number of such plane partitions is given by the following triple product formula, which then also counts lozenge tilings of $H(a,b,c)$, 
\[\prod_{i=1}^a\prod_{j=1}^b\prod_{k=1}^c\frac{i+j+k-1}{i+j+k-2}.\]

In this paper, we enumerate lozenge tilings of the triangular region $\mathcal{T}_n$ in Figure~\ref{fig:intro}. This problem was posed by Ciucu and Krattenthaler~{\cite[Problem 1.5]{CK}} in 2002 after they successfully enumerated lozenge tilings of several regions related to $H(a,b,c)$, and is included by Lai~{\cite[Problem 24]{L}} in his recent list of open problems in tiling enumeration. As observed by Ciucu and Krattenthaler, and presented in Table~\ref{table:1}, the number of lozenge tilings of $\mathcal{T}_n$, denoted by $T_n$, contains a substantial number of prime factors, which usually indicates the existence of a nice product formula. 
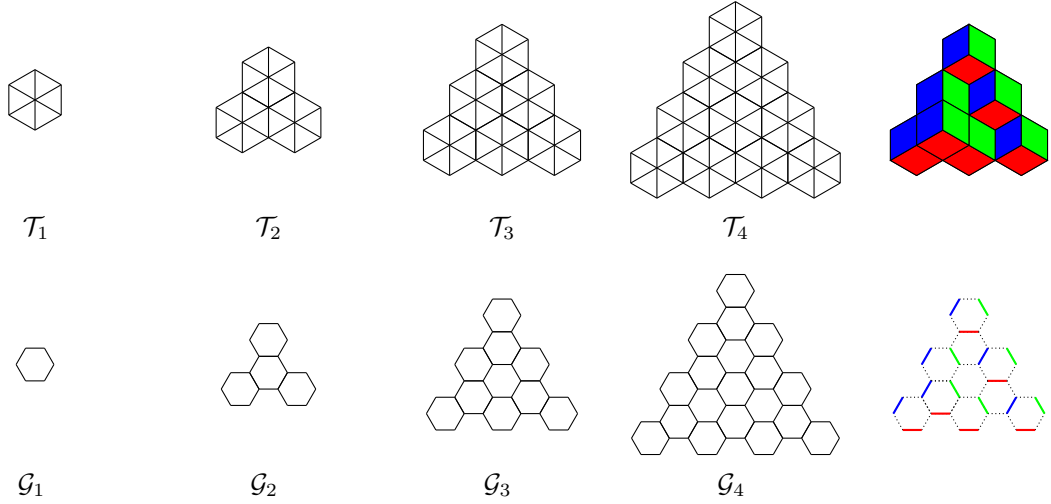
\begin{figure}[h]
    \centering
    \newcommand\hex{}
\def\hex[#1](#2:#3){
\draw[#1] ($({#2*sqrt(3)},#3)$) -- ($({(#2+0.25)*sqrt(3)},{0.25+#3})$);
\draw[#1] ($({#2*sqrt(3)},#3)$) -- ($({(#2-0.25)*sqrt(3)},{0.25+#3})$);
\draw[#1] ($({(#2-0.25)*sqrt(3)},{0.25+#3})$) -- ($({(#2-0.25)*sqrt(3)},{0.75+#3})$);
\draw[#1] ($({(#2+0.25)*sqrt(3)},{0.25+#3})$) -- ($({(#2+0.25)*sqrt(3)},{0.75+#3})$);
\draw[#1] ($({#2*sqrt(3)},{1+#3})$) -- ($({(#2+0.25)*sqrt(3)},{0.75+#3})$);
\draw[#1] ($({#2*sqrt(3)},{1+#3})$) -- ($({(#2-0.25)*sqrt(3)},{0.75+#3})$);
\draw[#1] ($({#2*sqrt(3)},{1+#3})$) -- ($({#2*sqrt(3)},{#3})$);
\draw[#1] ($({(#2+0.25)*sqrt(3)},{0.25+#3})$) -- ($({(#2-0.25)*sqrt(3)},{0.75+#3})$);
\draw[#1] ($({(#2-0.25)*sqrt(3)},{0.25+#3})$) -- ($({(#2+0.25)*sqrt(3)},{0.75+#3})$);
}

\newcommand\xeh{}
\def\xeh[#1](#2:#3){
\draw[#1] ($(#2,{#3*sqrt(3)})$) -- ($(#2+0.5,{#3*sqrt(3)})$);
\draw[#1] ($(#2,{(#3+0.5)*sqrt(3)})$) -- ($(#2+0.5,{(#3+0.5)*sqrt(3)})$);
\draw[#1] ($(#2,{#3*sqrt(3)})$) -- ($(#2-0.25,{(#3+0.25)*sqrt(3)})$);
\draw[#1] ($(#2,{(#3+0.5)*sqrt(3)})$) -- ($(#2-0.25,{(#3+0.25)*sqrt(3)})$);
\draw[#1] ($(#2+0.5,{#3*sqrt(3)})$) -- ($(#2+0.75,{(#3+0.25)*sqrt(3)})$);
\draw[#1] ($(#2+0.5,{(#3+0.5)*sqrt(3)})$) -- ($(#2+0.75,{(#3+0.25)*sqrt(3)})$);
}

\begin{minipage}{0.18\textwidth}
\begin{center}
\begin{tikzpicture}[scale=0.8]
\hex[](0.75:0.25);
\end{tikzpicture}
\end{center}
\end{minipage}
\begin{minipage}{0.18\textwidth}
\begin{center}
\begin{tikzpicture}[scale=0.8]
\hex[](0.5:-0.5);
\hex[](1:-0.5);
\hex[](0.75:0.25);
\end{tikzpicture}
\end{center}
\end{minipage}
\begin{minipage}{0.18\textwidth}
\begin{center}
\begin{tikzpicture}[scale=0.8]
\hex[](0.25:-1.25);
\hex[](0.75:-1.25);
\hex[](1.25:-1.25);
\hex[](0.5:-0.5);
\hex[](1:-0.5);
\hex[](0.75:0.25);
\end{tikzpicture}
\end{center}
\end{minipage}
\begin{minipage}{0.18\textwidth}
\begin{center}
\begin{tikzpicture}[scale=0.8]
\hex[](0:-2);
\hex[](0.5:-2);
\hex[](1:-2);
\hex[](1.5:-2);
\hex[](0.25:-1.25);
\hex[](0.75:-1.25);
\hex[](1.25:-1.25);
\hex[](0.5:-0.5);
\hex[](1:-0.5);
\hex[](0.75:0.25);
\end{tikzpicture}
\end{center}
\end{minipage}
\begin{minipage}{0.18\textwidth}
\begin{center}
\begin{tikzpicture}[scale=0.8]
\hex[](0.25:-1.25);
\hex[](0.75:-1.25);
\hex[](1.25:-1.25);
\hex[](0.5:-0.5);
\hex[](1:-0.5);
\hex[](0.75:0.25);

\filldraw [fill=red,draw] ($({1*sqrt(3)},{-0.5})$) -- ($({1.25*sqrt(3)},{-0.25})$) -- ($({1*sqrt(3)},{0})$) -- ($({0.75*sqrt(3)},{-0.25})$) -- cycle;
\filldraw [fill=red,draw] ($({0.75*sqrt(3)},{0.25})$) -- ($({1*sqrt(3)},{0.5})$) -- ($({0.75*sqrt(3)},{0.75})$) -- ($({0.5*sqrt(3)},{0.5})$) -- cycle;
\filldraw [fill=red,draw] ($({0.75*sqrt(3)},{-1.25})$) -- ($({1*sqrt(3)},{-1})$) -- ($({0.75*sqrt(3)},{-0.75})$) -- ($({0.5*sqrt(3)},{-1})$) -- cycle;
\filldraw [fill=red,draw] ($({0.25*sqrt(3)},{-1.25})$) -- ($({0.5*sqrt(3)},{-1})$) -- ($({0.25*sqrt(3)},{-0.75})$) -- ($({0*sqrt(3)},{-1})$) -- cycle;
\filldraw [fill=red,draw] ($({1.25*sqrt(3)},{-1.25})$) -- ($({1.5*sqrt(3)},{-1})$) -- ($({1.25*sqrt(3)},{-0.75})$) -- ($({1*sqrt(3)},{-1})$) -- cycle;
\filldraw [fill=red,draw] ($({0.5*sqrt(3)},{-1})$) -- ($({0.75*sqrt(3)},{-0.75})$) -- ($({0.5*sqrt(3)},{-0.5})$) -- ($({0.25*sqrt(3)},{-0.75})$) -- cycle;

\filldraw [fill=blue,draw] ($({0.25*sqrt(3)},{-0.25})$) -- ($({0.25*sqrt(3)},{0.25})$) -- ($({0.5*sqrt(3)},{0.5})$) -- ($({0.5*sqrt(3)},{0})$) -- cycle;
\filldraw [fill=blue,draw] ($({0.25*sqrt(3)},{-0.75})$) -- ($({0.25*sqrt(3)},{-0.25})$) -- ($({0.5*sqrt(3)},{0})$) -- ($({0.5*sqrt(3)},{-0.5})$) -- cycle;
\filldraw [fill=blue,draw] ($({0*sqrt(3)},{-1})$) -- ($({0*sqrt(3)},{-0.5})$) -- ($({0.25*sqrt(3)},{-0.25})$) -- ($({0.25*sqrt(3)},{-0.75})$) -- cycle;
\filldraw [fill=blue,draw] ($({0.75*sqrt(3)},{-0.25})$) -- ($({0.75*sqrt(3)},{0.25})$) -- ($({1*sqrt(3)},{0.5})$) -- ($({1*sqrt(3)},{0})$) -- cycle;
\filldraw [fill=blue,draw] ($({0.5*sqrt(3)},{0.5})$) -- ($({0.5*sqrt(3)},{1})$) -- ($({0.75*sqrt(3)},{1.25})$) -- ($({0.75*sqrt(3)},{0.75})$) -- cycle;
\filldraw [fill=blue,draw] ($({1*sqrt(3)},{-1})$) -- ($({1*sqrt(3)},{-0.5})$) -- ($({1.25*sqrt(3)},{-0.25})$) -- ($({1.25*sqrt(3)},{-0.75})$) -- cycle;

\filldraw [fill=green,draw] ($({1*sqrt(3)},{0})$) -- ($({1*sqrt(3)},{0.5})$) -- ($({1.25*sqrt(3)},{0.25})$) -- ($({1.25*sqrt(3)},{-0.25})$) -- cycle;
\filldraw [fill=green,draw] ($({0.5*sqrt(3)},{0})$) -- ($({0.5*sqrt(3)},{0.5})$) -- ($({0.75*sqrt(3)},{0.25})$) -- ($({0.75*sqrt(3)},{-0.25})$) -- cycle;
\filldraw [fill=green,draw] ($({0.75*sqrt(3)},{0.75})$) -- ($({0.75*sqrt(3)},{1.25})$) -- ($({1*sqrt(3)},{1})$) -- ($({1*sqrt(3)},{0.5})$) -- cycle;
\filldraw [fill=green,draw] ($({0.75*sqrt(3)},{-0.75})$) -- ($({0.75*sqrt(3)},{-0.25})$) -- ($({1*sqrt(3)},{-0.5})$) -- ($({1*sqrt(3)},{-1})$) -- cycle;
\filldraw [fill=green,draw] ($({1.25*sqrt(3)},{-0.75})$) -- ($({1.25*sqrt(3)},{-0.25})$) -- ($({1.5*sqrt(3)},{-0.5})$) -- ($({1.5*sqrt(3)},{-1})$) -- cycle;
\filldraw [fill=green,draw] ($({0.5*sqrt(3)},{-0.5})$) -- ($({0.5*sqrt(3)},{0})$) -- ($({0.75*sqrt(3)},{-0.25})$) -- ($({0.75*sqrt(3)},{-0.75})$) -- cycle;
\end{tikzpicture}
\end{center}
\end{minipage}

\medskip

\begin{minipage}{0.18\textwidth}
\begin{center}
$\mathcal{T}_1$
\end{center}
\end{minipage}
\begin{minipage}{0.18\textwidth}
\begin{center}
$\mathcal{T}_2$
\end{center}
\end{minipage}
\begin{minipage}{0.18\textwidth}
\begin{center}
$\mathcal{T}_3$
\end{center}
\end{minipage}
\begin{minipage}{0.18\textwidth}
\begin{center}
$\mathcal{T}_4$
\end{center}
\end{minipage}
\begin{minipage}{0.18\textwidth}
\begin{center}

\end{center}
\end{minipage}


\medskip

\medskip

\begin{minipage}{0.18\textwidth}
\begin{center}
\begin{tikzpicture}[scale=0.5]
\xeh[](0:0);
\end{tikzpicture}
\end{center}
\end{minipage}
\begin{minipage}{0.18\textwidth}
\begin{center}
\begin{tikzpicture}[scale=0.5]
\xeh[](0:0);
\xeh[](0:-0.5);
\xeh[](-0.75:-0.75);
\xeh[](0.75:-0.75);
\end{tikzpicture}
\end{center}
\end{minipage}
\begin{minipage}{0.18\textwidth}
\begin{center}
\begin{tikzpicture}[scale=0.5]
\xeh[](0:0);
\xeh[](0:-0.5);
\xeh[](-0.75:-0.75);
\xeh[](0.75:-0.75);
\xeh[](-0.75:-1.25);
\xeh[](-1.5:-1.5);
\xeh[](0:-1.5);
\xeh[](0.75:-1.25);
\xeh[](1.5:-1.5);
\end{tikzpicture}
\end{center}
\end{minipage}
\begin{minipage}{0.18\textwidth}
\begin{center}
\begin{tikzpicture}[scale=0.5]
\xeh[](0:0);
\xeh[](0:-0.5);
\xeh[](-0.75:-0.75);
\xeh[](0.75:-0.75);
\xeh[](-0.75:-1.25);
\xeh[](0.75:-1.25);
\xeh[](-1.5:-1.5);
\xeh[](0:-1.5);
\xeh[](1.5:-1.5);
\xeh[](-1.5:-2);
\xeh[](0:-2);
\xeh[](1.5:-2);
\xeh[](-2.25:-2.25);
\xeh[](-0.75:-2.25);
\xeh[](0.75:-2.25);
\xeh[](2.25:-2.25);
\end{tikzpicture}
\end{center}
\end{minipage}
\begin{minipage}{0.18\textwidth}
\begin{center}
\begin{tikzpicture}[scale=0.5]
\xeh[thin, densely dotted](0:0);
\xeh[thin, densely dotted](0:-0.5);
\xeh[thin, densely dotted](-0.75:-0.75);
\xeh[thin, densely dotted](0.75:-0.75);
\xeh[thin, densely dotted](-0.75:-1.25);
\xeh[thin, densely dotted](0.75:-1.25);
\xeh[thin, densely dotted](-1.5:-1.5);
\xeh[thin, densely dotted](0:-1.5);
\xeh[thin, densely dotted](1.5:-1.5);

\draw[red,thick] ($(1.5,{-1.5*sqrt(3)})$) -- ($(2,{-1.5*sqrt(3)})$);
\draw[red,thick] ($(0,{-1.5*sqrt(3)})$) -- ($(0.5,{-1.5*sqrt(3)})$);
\draw[red,thick] ($(-1.5,{-1.5*sqrt(3)})$) -- ($(-1,{-1.5*sqrt(3)})$);
\draw[red,thick] ($(-0.75,{-1.25*sqrt(3)})$) -- ($(-0.25,{-1.25*sqrt(3)})$);
\draw[red,thick] ($(0,{0*sqrt(3)})$) -- ($(0.5,{0*sqrt(3)})$);
\draw[red,thick] ($(0.75,{-0.75*sqrt(3)})$) -- ($(1.25,{-0.75*sqrt(3)})$);

\draw[green,thick] ($(2,{-1*sqrt(3)})$) -- ($(2.25,{-1.25*sqrt(3)})$);
\draw[green,thick] ($(1.25,{-0.25*sqrt(3)})$) -- ($(1.5,{-0.5*sqrt(3)})$);
\draw[green,thick] ($(0.5,{0.5*sqrt(3)})$) -- ($(0.75,{0.25*sqrt(3)})$);
\draw[green,thick] ($(-0.25,{-0.25*sqrt(3)})$) -- ($(0,{-0.5*sqrt(3)})$);
\draw[green,thick] ($(0.5,{-1*sqrt(3)})$) -- ($(0.75,{-1.25*sqrt(3)})$);
\draw[green,thick] ($(-0.25,{-0.75*sqrt(3)})$) -- ($(0,{-1*sqrt(3)})$);

\draw[blue,thick] ($(1.25,{-1.25*sqrt(3)})$) -- ($(1.5,{-1*sqrt(3)})$);
\draw[blue,thick] ($(-1.75,{-1.25*sqrt(3)})$) -- ($(-1.5,{-1*sqrt(3)})$);
\draw[blue,thick] ($(-1,{-0.5*sqrt(3)})$) -- ($(-0.75,{-0.25*sqrt(3)})$);
\draw[blue,thick] ($(-0.25,{0.25*sqrt(3)})$) -- ($(0,{0.5*sqrt(3)})$);
\draw[blue,thick] ($(-1,{-1*sqrt(3)})$) -- ($(-0.75,{-0.75*sqrt(3)})$);
\draw[blue,thick] ($(0.5,{-0.5*sqrt(3)})$) -- ($(0.75,{-0.25*sqrt(3)})$);
\end{tikzpicture}
\end{center}
\end{minipage}

\medskip

\begin{minipage}{0.18\textwidth}
\begin{center}
$\mathcal{G}_1$
\end{center}
\end{minipage}
\begin{minipage}{0.18\textwidth}
\begin{center}
$\mathcal{G}_2$
\end{center}
\end{minipage}
\begin{minipage}{0.18\textwidth}
\begin{center}
$\mathcal{G}_3$
\end{center}
\end{minipage}
\begin{minipage}{0.18\textwidth}
\begin{center}
$\mathcal{G}_4$
\end{center}
\end{minipage}
\begin{minipage}{0.18\textwidth}
\begin{center}

\end{center}
\end{minipage}
    \caption{On the left, the triangular regions $\mathcal{T}_n$ for $n=1,2,3,4$, and the corresponding graphs $\mathcal{G}_n$ for $n=1,2,3,4$. On the right, a lozenge tiling of $\mathcal{T}_3$, and the corresponding perfect matching in $\mathcal{G}_3$.}\label{fig:intro}
\end{figure}

\begin{table}[h]
\renewcommand{\arraystretch}{1.25}
\begin{center}
\begin{tabular}{c||c|c|c|c|c|c|c}
 $n$ & 1 & 2 & 3 & 4 & 5 & 6 & 7 \\\hline
 $T_n$ & 2 & 9 & 104 & 3100 & 240426 & 48701198 & 25827984000 \\\hline
 factorisation & 2 & $3^2$ & $2^3\cdot13$ & $2^2\cdot 5^2\cdot 31$ & $2\cdot3^2\cdot19^2\cdot37$ & $2\cdot 7^3\cdot13\cdot43\cdot127$ & $2^7\cdot3^5\cdot5^3\cdot7\cdot13\cdot73$ \\   
\end{tabular}
\end{center}
\caption{Values of $T_n$ for $1\leq n\leq 7$, and their prime factorisations.}\label{table:1}
\end{table}

Our main result, Theorem~\ref{thm:main} below, provides an exact product formula for $T_n$, the number of lozenge tilings of $\mathcal{T}_n$. Instead of counting lozenge tilings directly, we count perfect matchings of the graph $\mathcal{G}_n$ depicted in Figure~\ref{fig:intro}. This is obtained by taking the dual of the underlying planar graph of the region $\mathcal{T}_n$, so that vertices in $\mathcal{G}_n$ correspond to triangular faces in $\mathcal{T}_n$, and two vertices in $\mathcal{G}_n$ are connected with an edge if and only if the two corresponding triangles in $\mathcal{T}_n$ can be covered by a lozenge. It follows that lozenge tilings of $\mathcal{T}_n$ are in bijection with perfect matchings of $\mathcal{G}_n$.  
\begin{theorem}\label{thm:main}
For every positive integer $n$, let $\zeta=e^{2\pi i/(3n+3)}$. Then, the number of lozenge tilings of the region $\mathcal{T}_n$ and the number of perfect matchings of the graph $\mathcal{G}_n$ are both
\[T_n=\prod_{\substack{1\leq a<b\leq 3n+2\\(a,b)\not=(n+1,2n+2)}}\left|1+\zeta^a+\zeta^b\right|^{1/3}.\]
\end{theorem}

At first glance, it is not even obvious that the above formula for $T_n$ produces an integer, but this can be proved using Galois Theory. In Section~\ref{sec:factor}, we show that the formula actually explains the extensive factorisations of $T_n$ through the following result. The $n/5$ bound in~\ref{factor4} is likely far from optimal, but we did not pursue further improvement as it is not the main focus of the paper.
\begin{theorem}\label{thm:factor}
The numbers $T_n$ satisfy the following. 
\begin{enumerate}[label=\textup{\alph*)}]
    \item\label{factor1} For each positive integer $n$, let $\zeta=e^{2\pi i/(3n+3)}$, and define a polynomial $P_n(x)$ by 
    \[P_n(x)=\prod_{\substack{1\leq a<b\leq 3n+2\\(a,b)\not=(n+1,2n+2)}}(x+\zeta^a+\zeta^b).\]
    Then, $P_n(x)\in\mathbb{Z}[x]$, $P_n(x)$ can be factored in $\mathbb{Z}[x]$ into polynomials with degrees at most $2n+2$, and $P_n(1)=\pm T_n^3$.
    \item\label{factor2} If $m,n$ are positive integers satisfying $m+1\mid n+1$, then $T_m\mid T_n$.
    \item\label{factor3} Let $\chi_{-3}$ be the primitive quadratic character modulo 3, and let $L(2,\chi_{-3})$ be the Dirichlet $L$-function for $\chi_{-3}$ evaluated at $2$. Then, \[T_n=\exp\left(\frac{9\sqrt 3}{8\pi}L(2,\chi_{-3})n^2+o(n^2)\right)\approx(1.6235)^{n^2}.\]
    \item\label{factor4} For all sufficiently large $n$, $T_n$ contains at least $n/5$ prime factors, counting multiplicity.
\end{enumerate}
\end{theorem}

We briefly mention two immediate consequences of Theorem~\ref{thm:main} using well-known bijections between lozenge tilings and plane partitions, and between lozenge tilings and non-intersecting lattice paths. 

A plane partition with a given Young diagram shape is a filling of the Young diagram with non-negative integers so that each row and each column is non-increasing. Geometrically, this can be visualised as a stacking of unit cubes on top of the Young diagram so that the numbers of cubes are non-increasing along each row and each column. Using the same idea as David and Tomei~\cite{DT}, if we view the red lozenges in each lozenge tiling of $\mathcal{T}_n$ as the upper faces of stacked unit cubes (see Figure~\ref{fig:parLGV}), then this leads to a bijection between lozenge tilings of $\mathcal{T}_n$ and plane partitions $\pi$ with shape $(n,n-1,\ldots,1)$, such that $\pi(i,j)\leq n+1-\max\{i,j\}$ for all integers $i,j\geq1$ with $i+j\leq n+1$. This gives the following corollary. 
\begin{corollary}\label{cor:PP}
For every positive integer $n$, let $\zeta=e^{2\pi i/(3n+3)}$. Then, the number of plane partitions $\pi$ with shape $(n,n-1,\ldots,1)$, such that $\pi(i,j)\leq n+1-\max\{i,j\}$ for all integers $i,j\geq1$ with $i+j\leq n+1$ is 
\[T_n=\prod_{\substack{1\leq a<b\leq 3n+2\\(a,b)\not=(n+1,2n+2)}}\left|1+\zeta^a+\zeta^b\right|^{1/3}.\]
\end{corollary}

\begin{figure}[h]
\centering
    \input{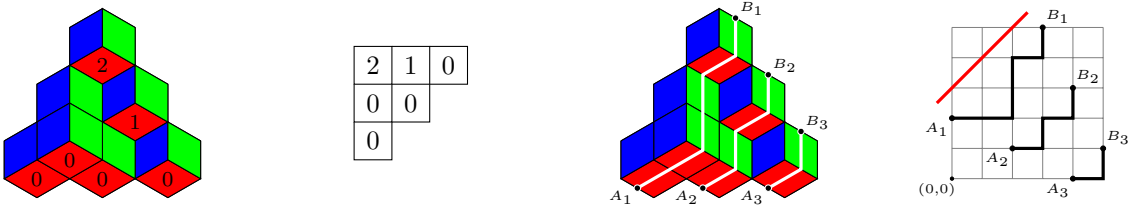}
    \caption{On the left, a lozenge tiling of $\mathcal{T}_3$ and its corresponding plane partition with shape $(3,2,1)$. On the right, the same lozenge tiling of $\mathcal{T}_3$ and its corresponding triple of non-intersecting lattice paths, which do not go above the red line $y=x+3$.}\label{fig:parLGV}
\end{figure}

Using the same cube stacking visualisation, now imagine $n$ bugs trying to crawl from the bottom-left edges to the top-right edges along the stacked cubes. Each valid cube stacking gives rise to an $n$-tuple of disjoint crawl paths. After suitable geometric transformations, these crawl paths turn out to be in bijection with $n$-tuples of non-intersecting lattice paths $(Q_1,\ldots,Q_n)$ that do not go above the line $y=x+n$, such that $Q_i$ goes from $(2i-2,n-i)$ to $(n+i-1,2n+1-2i)$ for every $i\in[n]$. Such collections of non-intersecting lattice paths can be counted by evaluating a determinant using the following version of the Lindstr\"om-Gessel-Viennot Lemma, originating from the work of Lindstr\"om~\cite{Lind} in 1973, and of Gessel and Viennot~\cite{GV} in 1985.
\begin{lemma}[Lindstr\"om-Gessel-Viennot Lemma {\cite[Corollary 10.13.2]{Kra}}]\label{lemma:LGV}
Let $n$ be a positive integer. Let $G$ be a finite directed acyclic graph. Let $A_1,\ldots,A_n,B_1,\ldots,B_n$ be distinct vertices in $G$. Suppose there does not exist a permutation $\sigma\in S_n\setminus\{\textup{id}\}$ such that there are vertex-disjoint directed paths $Q_1,\ldots,Q_n$, with $Q_i$ being a directed path in $G$ from $A_i$ to $B_{\sigma(i)}$ for every $i\in[n]$.

Let $\mathcal{Q}$ be the set of vertex-disjoint tuples of directed paths $(Q_1,\ldots,Q_n)$, such that $Q_i$ is a directed path from $A_i$ to $B_i$ for every $i\in[n]$. Let $D$ be the $n\times n$ matrix whose $(i,j)$-entry is the number of directed paths from $A_i$ to $B_j$ in $G$. Then,
\[|\mathcal{Q}|=\det(D).\]
\end{lemma}

For every $i,j\in[n]$, it is known~{\cite[Theorem 10.3.1]{Kra}} that the number of lattice paths from $(2i-2,n-i)$ to $(n+j-1,2n+1-2j)$ that do not go above $y=x+n$ is $\binom{2n+2-i-j}{n-2i+j+1}-\binom{2n+2-i-j}{n+i+j}$. Thus, Theorem~\ref{thm:main} has the following immediate corollary. 
\begin{corollary}\label{cor:LGV}
For every positive integer $n$, let $\zeta=e^{2\pi i/(3n+3)}$. Let $D_n$ be the $n\times n$ matrix whose $(i,j)$-entry is $\binom{2n+2-i-j}{n-2i+j+1}-\binom{2n+2-i-j}{n+i+j}$ for every $i,j\in[n]$. Then, the number of $n$-tuples of non-intersecting lattice paths $(Q_1,\ldots,Q_n)$ that do not go above $y=x+n$, with $Q_i$ going from $(2i-2,n-i)$ to $(n+i-1,2n+1-2i)$ for each $i\in[n]$ is 
\[\det(D_n)=T_n=\prod_{\substack{1\leq a<b\leq 3n+2\\(a,b)\not=(n+1,2n+2)}}\left|1+\zeta^a+\zeta^b\right|^{1/3}.\]
\end{corollary}

It would be interesting to know whether the determinant of the matrix $D_n$ can be evaluated in a more direct way, which would give an alternative proof of Theorem~\ref{thm:main}. 

\section{Main Result}
In this section, we prove our main result, Theorem~\ref{thm:main}. We begin by providing a proof sketch.
\subsection{Proof sketch}\label{sec:sketch}
The first step in our proof of Theorem~\ref{thm:main} is to transform the perfect matching enumeration problem into a determinant evaluation problem. To motivate the main idea, we start with the following lemma that turns counting perfect matchings in a balanced bipartite graph into evaluating the permanent of the bipartite adjacency matrix. The proof is immediate after unravelling the definitions.
\begin{lemma}\label{lemma:perm}
Let $G$ be a balanced bipartite graph with bipartition classes $\{a_1,\ldots,a_n\}$ and $\{b_1,\ldots,b_n\}$. Let $M$ be the $n\times n$ bipartite adjacency matrix of $G$, so that for any $i,j\in[n]$, $M(i,j)=1$ if $a_ib_j\in E(G)$, and $M(i,j)=0$ otherwise. Then, the number of perfect matchings in $G$ is given by the permanent of $M$, which is
\[\per(M)=\sum_{\sigma\in S_n}\prod_{i=1}^nM(i,\sigma(i)).\]
\end{lemma}

While Lemma~\ref{lemma:perm} is simple to state and prove, it has the drawback that not many methods exist for evaluating permanents. In contrast, there are many ways to evaluate determinants, but the additional $\sgn(\sigma)$ term in the definition of determinants prevents a direct translation of Lemma~\ref{lemma:perm} to determinants in general. To combat this, we use an idea originating from the aforementioned work of Kasteleyn~\cite{K} and Temperley and Fisher~\cite{TF}, which is to replace the 1 entries in $M$ with some suitable edge weights with modulus 1, so that they cancel out the effects of the $\sgn(\sigma)$ terms. The following version for planar bipartite graphs in the lecture notes of Kenyon~\cite{Ken} suffices for our purpose, though there is a more general version of Kasteleyn~\cite{K2} for non-bipartite planar graphs as well.
\begin{lemma}[{\cite[Theorem 2]{Ken}}]\label{lemma:weight}
Let $G$ be a planar bipartite graph with bipartition classes $\{a_1,\ldots,a_n\}$ and $\{b_1,\ldots,b_n\}$. Suppose $w:[n]^2\to\{0,\pm1\}$ is a weight function such that $w(i,j)\not=0$ if and only if $a_ib_j$ is an edge in $G$. Moreover, suppose that every bounded face of $G$ with size divisible by 4 contains an odd number of edges with weight $-1$, while every bounded face with size not divisible by 4 contains an even number of edges with weight $-1$. Then, the number of perfect matchings in $G$ is equal to $|\det(M)|$ for the matrix $M$ with $M(i,j)=w(i,j)$ for every $i,j\in[n]$.
\end{lemma}

Since $\mathcal{G}_n$ is a bipartite planar graph whose bounded faces are all hexagons, we have the following immediate corollary, as Lemma~\ref{lemma:weight} implies that simply giving each edge weight 1 works. 
\begin{corollary}\label{cor:det=}
Let $M_n$ be a bipartite adjacency matrix of $\mathcal{G}_n$ with the rows indexed by one bipartition class of $\mathcal{G}_n$ and the columns indexed by the other. Then, the number of perfect matchings in $\mathcal{G}_n$ is $|\det(M_n)|$.
\end{corollary}

To evaluate the determinant of $M_n$, we use a strategy employed by Kenyon, Propp, and Wilson in~{\cite[Section 6.9]{KPW}} to find the number of perfect matchings of a similar graph with mostly hexagonal faces.

We place $\mathcal{G}_n$ on the complex plane as shown in Figure~\ref{fig:coordinate}, so that the vertices all correspond to points in $\mathbb{Z}[\omega]$. Vertices in the two bipartition classes of $\mathcal{G}_n$, say $B$ and $W$, are coloured with black and white, and referred to as black and white vertices, respectively. Up to translation, this allows us to label each vertex at $i+j\omega+k\omega^2$ with the coordinate $(i,j,k)$. The benefit of this labelling is that the edge adjacencies in $\mathcal{G}_n$ can now be described in a simple and symmetric way. For example, for each black vertex with coordinate $(i,j,k)$, its neighbours are those with coordinates $(i+1,j,k), (i,j+1,k), (i,j,k+1)$ that are in $\mathcal{G}_n$. 

\begin{figure}[h]
    \centering
    \newcommand\xeh{}
\def\xeh[#1](#2:#3){
\draw[#1] ($(#2,{#3*sqrt(3)})$) -- ($(#2+0.5,{#3*sqrt(3)})$);
\draw[#1] ($(#2,{(#3+0.5)*sqrt(3)})$) -- ($(#2+0.5,{(#3+0.5)*sqrt(3)})$);
\draw[#1] ($(#2,{#3*sqrt(3)})$) -- ($(#2-0.25,{(#3+0.25)*sqrt(3)})$);
\draw[#1] ($(#2,{(#3+0.5)*sqrt(3)})$) -- ($(#2-0.25,{(#3+0.25)*sqrt(3)})$);
\draw[#1] ($(#2+0.5,{#3*sqrt(3)})$) -- ($(#2+0.75,{(#3+0.25)*sqrt(3)})$);
\draw[#1] ($(#2+0.5,{(#3+0.5)*sqrt(3)})$) -- ($(#2+0.75,{(#3+0.25)*sqrt(3)})$);
}

\newcommand\xehp{}
\def\xehp[#1](#2:#3){
\draw[#1] ($(#2,{#3*sqrt(3)})$) -- ($(#2+0.5,{#3*sqrt(3)})$);
\draw[#1] ($(#2,{(#3+0.5)*sqrt(3)})$) -- ($(#2+0.5,{(#3+0.5)*sqrt(3)})$);
\draw[#1] ($(#2,{#3*sqrt(3)})$) -- ($(#2-0.25,{(#3+0.25)*sqrt(3)})$);
\draw[#1] ($(#2,{(#3+0.5)*sqrt(3)})$) -- ($(#2-0.25,{(#3+0.25)*sqrt(3)})$);
\draw[#1] ($(#2+0.5,{#3*sqrt(3)})$) -- ($(#2+0.75,{(#3+0.25)*sqrt(3)})$);
\draw[#1] ($(#2+0.5,{(#3+0.5)*sqrt(3)})$) -- ($(#2+0.75,{(#3+0.25)*sqrt(3)})$);

\node[draw=black,inner sep=0.4ex,circle,fill=black] at ($(#2,{#3*sqrt(3)})$) {};
\node[draw=black,inner sep=0.4ex,circle,fill=black] at ($(#2,{(#3+0.5)*sqrt(3)})$) {};
\node[draw=black,inner sep=0.4ex,circle,fill=black] at ($(#2+0.75,{(#3+0.25)*sqrt(3)})$) {};

\pgfmathtruncatemacro{\a}{#2*4/3+4}
\pgfmathtruncatemacro{\b}{-#2*2/3+#3*2+2}
\pgfmathtruncatemacro{\c}{-#2*2/3-#3*2-5}
\pgfmathtruncatemacro{\d}{#2*4/3+5}
\pgfmathtruncatemacro{\e}{-#2*2/3+#3*2+2}
\pgfmathtruncatemacro{\f}{-#2*2/3-#3*2-6}
\pgfmathtruncatemacro{\g}{#2*4/3+4}
\pgfmathtruncatemacro{\h}{-#2*2/3+#3*2+3}
\pgfmathtruncatemacro{\i}{-#2*2/3-#3*2-6}
\node[draw=none, fill=none, below] at ($(#2,{#3*sqrt(3)})$) {\tiny (\a,\b,\c)};
\node[draw=none, fill=none, below] at ($(#2,{(#3+0.5)*sqrt(3)})$) {\tiny (\g,\h,\i)};
\node[draw=none, fill=none, below] at ($(#2+0.75,{(#3+0.25)*sqrt(3)})$) {\tiny (\d,\e,\f)};

\node[draw=black,inner sep=0.4ex,circle,fill=white] at ($(#2+0.5,{#3*sqrt(3)})$) {};
\node[draw=black,inner sep=0.4ex,circle,fill=white] at ($(#2+0.5,{(#3+0.5)*sqrt(3)})$) {};
\node[draw=black,inner sep=0.4ex,circle,fill=white] at ($(#2-0.25,{(#3+0.25)*sqrt(3)})$) {};

\pgfmathtruncatemacro{\a}{#2*4/3+5}
\pgfmathtruncatemacro{\b}{-#2*2/3+#3*2+2}
\pgfmathtruncatemacro{\c}{-#2*2/3-#3*2-5}
\pgfmathtruncatemacro{\d}{#2*4/3+4}
\pgfmathtruncatemacro{\e}{-#2*2/3+#3*2+3}
\pgfmathtruncatemacro{\f}{-#2*2/3-#3*2-5}
\pgfmathtruncatemacro{\g}{#2*4/3+5}
\pgfmathtruncatemacro{\h}{-#2*2/3+#3*2+3}
\pgfmathtruncatemacro{\i}{-#2*2/3-#3*2-6}
\node[draw=none, fill=none, below] at ($(#2+0.5,{#3*sqrt(3)})$) {\tiny (\a,\b,\c)};
\node[draw=none, fill=none, below] at ($(#2+0.5,{(#3+0.5)*sqrt(3)})$) {\tiny (\g,\h,\i)};
\node[draw=none, fill=none, below] at ($(#2-0.25,{(#3+0.25)*sqrt(3)})$) {\tiny (\d,\e,\f)};
}

\begin{tikzpicture}[scale=1.75]

\xeh[densely dotted](0:-0.5);
\xeh[densely dotted](-0.75:-1.25);
\xeh[densely dotted](0.75:-1.25);

\xehp[densely dotted](0:0);
\xehp[densely dotted](-0.75:-0.75);
\xehp[densely dotted](0.75:-0.75);
\xehp[densely dotted](-1.5:-1.5);
\xehp[densely dotted](0:-1.5);
\xehp[densely dotted](1.5:-1.5);

\draw[densely dotted] ($(-2.75,{-1.75*sqrt(3)})$) -- ($(3.25,{-1.75*sqrt(3)})$) -- ($(0.25,{1.25*sqrt(3)})$) -- ($(-2.75,{-1.75*sqrt(3)})$);

\node[draw=black,inner sep=0.4ex,circle,fill=red] at ($(1.25,{0.25*sqrt(3)})$) {};
\node[draw=black,inner sep=0.4ex,circle,fill=red] at ($(2,{-0.5*sqrt(3)})$) {};
\node[draw=black,inner sep=0.4ex,circle,fill=red] at ($(2.75,{-1.25*sqrt(3)})$) {};
\node[draw=none,fill=none,below] at ($(1.25,{0.25*sqrt(3)})$) {\tiny (6,2,-6)};
\node[draw=none,fill=none,below] at ($(2,{-0.5*sqrt(3)})$) {\tiny (7,0,-5)};
\node[draw=none,fill=none,below] at ($(2.75,{-1.25*sqrt(3)})$) {\tiny (8,-2,-4)};

\node[draw=black,inner sep=0.4ex,circle,fill=red] at ($(-0.25,{0.75*sqrt(3)})$) {};
\node[draw=black,inner sep=0.4ex,circle,fill=red] at ($(-1,{0*sqrt(3)})$) {};
\node[draw=black,inner sep=0.4ex,circle,fill=red] at ($(-1.75,{-0.75*sqrt(3)})$) {};
\node[draw=none,fill=none,below] at ($(-0.25,{0.75*sqrt(3)})$) {\tiny (4,4,-6)};
\node[draw=none,fill=none,below] at ($(-1,{0*sqrt(3)})$) {\tiny (3,3,-4)};
\node[draw=none,fill=none,below] at ($(-1.75,{-0.75*sqrt(3)})$) {\tiny (2,2,-2)};

\node[draw=black,inner sep=0.4ex,circle,fill=red] at ($(-1.75,{-1.75*sqrt(3)})$) {};
\node[draw=black,inner sep=0.4ex,circle,fill=red] at ($(-0.25,{-1.75*sqrt(3)})$) {};
\node[draw=black,inner sep=0.4ex,circle,fill=red] at ($(1.25,{-1.75*sqrt(3)})$) {};
\node[draw=none,fill=none,below] at ($(-1.75,{-1.75*sqrt(3)})$) {\tiny (2,0,0)};
\node[draw=none,fill=none,below] at ($(-0.25,{-1.75*sqrt(3)})$) {\tiny (4,-1,-1)};
\node[draw=none,fill=none,below] at ($(1.25,{-1.75*sqrt(3)})$) {\tiny (6,-2,-2)};

\node[draw=none,fill=none] at ($(0.25,{-2*sqrt(3)})$) {$L_1:j=k$};
\node[draw=none,fill=none] at ($(2.75,{-0.25*sqrt(3)})$) {$L_2:i-k=12$};
\node[draw=none,fill=none] at ($(-2,{-0.25*sqrt(3)})$) {$L_3:i=j$};

\node[draw=black,inner sep=0.2ex,circle,fill=black] at ($(-2.75,{-1.75*sqrt(3)})$) {};
\node[draw=none,fill=none,below right = 0 cm and -0.2 cm] at ($(-2.75,{-1.75*sqrt(3)})$) {\tiny (0,0,0)};
\draw[->] ($(-2.75,{-1.75*sqrt(3)})$) -- ($(-2.25,{-1.75*sqrt(3)})$);
\node[draw=none,fill=none,below] at ($(-2.25,{-1.75*sqrt(3)})$) {\tiny 1};
\draw[->] ($(-2.75,{-1.75*sqrt(3)})$) -- ($(-3,{-1.5*sqrt(3)})$);
\node[draw=none,fill=none,above] at ($(-3,{-1.5*sqrt(3)})$) {\tiny $\omega$};
\draw[->] ($(-2.75,{-1.75*sqrt(3)})$) -- ($(-3,{-2*sqrt(3)})$);
\node[draw=none,fill=none,below] at ($(-3,{-2*sqrt(3)})$) {\tiny $\omega^2$};
\end{tikzpicture}
    \caption{The graph $\mathcal{G}_3$ placed on the complex plane. Each vertex is labelled with $(i,j,k)$ so that $i+j\omega+k\omega^2$ is the complex number corresponding to its location. Additionally, each black vertex satisfies $i+j+k=1$, and each white vertex satisfies $i+j+k=2$. Each red point would have been the neighbour of a black vertex, but is not because it lies on one of the boundary lines $L_1$, $L_2$, or $L_3$, and is not in $\mathcal{G}_n$.}\label{fig:coordinate}
\end{figure}
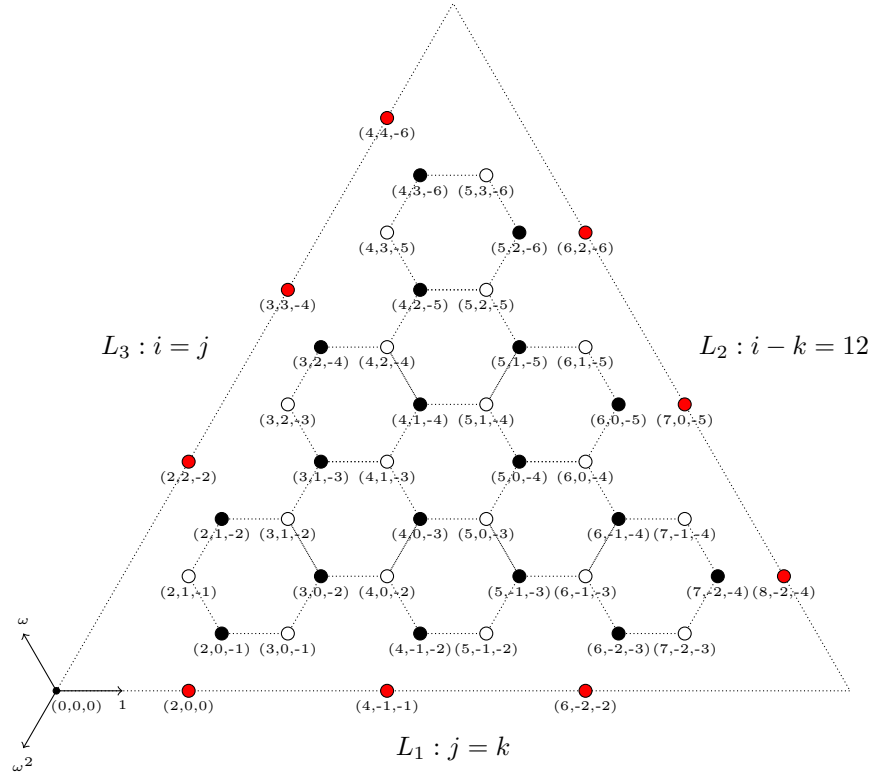

Then, we find two families of ``eigenvectors" $F^{\alpha,\beta,\gamma}_B\in\mathbb{C}^B$ and $F^{\alpha,\beta,\gamma}_W\in\mathbb{C}^W$ indexed by $(3n+3)$-th roots of unity $(\alpha,\beta,\gamma)$, such that $M_nF^{\alpha,\beta,\gamma}_W=\lambda^{\alpha,\beta,\gamma}F^{\alpha,\beta,\gamma}_B$ for some $\lambda^{\alpha,\beta,\gamma}\in\mathbb{C}$. The simple choices of $f^{\alpha,\beta,\gamma}_B(i,j,k)=\alpha^i\beta^j\gamma^k$ and $f^{\alpha,\beta,\gamma}_W(i,j,k)=\alpha^i\beta^j\gamma^k$ almost work, except for some irregularities on the boundary. To fix this, we let $F^{\alpha,\beta,\gamma}_W$ be a signed sum of the permuted versions of $f^{\alpha,\beta,\gamma}_W$, so that it vanishes on the boundary, and similarly for $F^{\alpha,\beta,\gamma}_B$. The corresponding ``eigenvalues" turn out to be $\lambda^{\alpha,\beta,\gamma}=\alpha+\beta+\gamma$. Of course, these are not actually eigenvectors and eigenvalues of the matrix $M_n$, as $F^{\alpha,\beta,\gamma}_B$ and $F^{\alpha,\beta,\gamma}_W$ are not the same vectors, but by doing this we have essentially diagonalised $M_n$. 

To make this diagonalisation precise, we show that if the triples of roots of unity $(\alpha,\beta,\gamma)$ are considered only up to scalar multiplication and permutation symmetry, and after excluding some degenerate triples, the remaining vectors $F^{\alpha,\beta,\gamma}_B$ and $F^{\alpha,\beta,\gamma}_W$ form bases of $\mathbb{C}^B$ and $\mathbb{C}^W$, respectively. Proving linear independence is surprisingly delicate, and requires further expanding the symmetry and enlarging the set of points considered, as depicted later in Figure~\ref{fig:symmetry}, so that the inner products of these vectors can be more easily computed. Assuming this, it follows that when represented in these new bases, the matrix $M_n$ becomes diagonal with diagonal entries $\alpha+\beta+\gamma$, for some suitable collection of triples $(\alpha,\beta,\gamma)$. The desired formula for $|\det(M_n)|$ and thus for $T_n$ follows after some additional algebraic manipulations.

\subsection{Proof of Theorem~\ref{thm:main}}
We now prove Theorem~\ref{thm:main} following the sketch in Section~\ref{sec:sketch}.
\begin{proof}[Proof of Theorem~\ref{thm:main}] \mbox{}

\medskip

\textbf{I. Setup and reduction to determinant evaluation}

Place the graph $\mathcal{G}_n$ on the complex plane as in Figure~\ref{fig:coordinate}. Let $\omega=e^{2\pi i/3}$, and recall that $1+\omega+\omega^2=0$. Then, each vertex corresponds to a complex number in $\mathbb{Z}[\omega]$, which can be represented non-uniquely as $i+j\omega+k\omega^2$ for some $i,j,k\in\mathbb{Z}$. Note that if $i,j,k,i',j',k'\in\mathbb{Z}$, then $i+j\omega+k\omega^2=i'+j'\omega+k'\omega^2$ if and only if $(i,j,k)=(i',j',k')+t(1,1,1)$ for some $t\in\mathbb{Z}$. Thus, we can label every point in $\mathbb{Z}[\omega]$, and in particular each black and white vertex in $\mathcal{G}_n$, uniquely as $(i,j,k)\in\mathbb{Z}^3$ with $i+j+k\in\{0,1,2\}$. 

The three boundary lines are 
\[L_1=\{(i,j,k):j=k\},\;\;L_2=\{(i,j,k):i-k=3n+3\},\;\;L_3=\{(i,j,k):i=j\}.\]
Each black vertex $(i,j,k)$ satisfies $i+j+k=1$, and is adjacent to white vertices $(i+1,j,k)$, $(i,j+1,k)$, and $(i,j,k+1)$, except possibly one which lies on the boundary. Each white vertex $(i,j,k)$ satisfies $i+j+k=2$, and is adjacent to $(i-1,j,k)$, $(i,j-1,k)$, and $(i,j,k-1)$, except possibly one which lies on the boundary.

Let $M_n$ be the bipartite adjacency matrix of $\mathcal{G}_n$, whose rows are indexed by the set $B$ of black vertices, whose columns are indexed by the set $W$ of white vertices, and $M_n((i,j,k),(i',j',k'))=1$ if $(i,j,k)(i',j',k')$ is an edge in $\mathcal{G}_n$, and 0 otherwise. Note that the dimension of $M_n$ is $3n(n+1)/2$. By Corollary~\ref{cor:det=}, the number of perfect matchings in $\mathcal{G}_n$ is equal to $|\det(M_n)|$.

\medskip

\textbf{II. Candidate vectors for diagonalisation}

For each triple of complex numbers $\alpha,\beta,\gamma$ satisfying $\alpha^{3n+3}=\beta^{3n+3}=\gamma^{3n+3}=1$, define a vector $f^{\alpha,\beta,\gamma}\in\mathbb{C}^{\mathbb{Z}^3}$ by $f^{\alpha,\beta,\gamma}(i,j,k)=\alpha^i\beta^j\gamma^k$. Denote the restrictions of $f^{\alpha,\beta,\gamma}$ to $B$ and $W$ by $f_B^{\alpha,\beta,\gamma}\in\mathbb{C}^B$ and $f_W^{\alpha,\beta,\gamma}\in\mathbb{C}^W$, respectively. Furthermore, define a vector $F^{\alpha,\beta,\gamma}\in\mathbb{C}^{\mathbb{Z}^3}$ as the following signed sum of the six superscript-permuted versions of $f^{\alpha,\beta,\gamma}$:
\[F^{\alpha,\beta,\gamma}=f^{\alpha,\beta,\gamma}-f^{\beta,\alpha,\gamma}-f^{\alpha,\gamma,\beta}-f^{\gamma,\beta,\alpha}+f^{\beta,\gamma,\alpha}+f^{\gamma,\alpha,\beta},\]
and denote its restriction to $B$ and $W$ by $F_B^{\alpha,\beta,\gamma}\in\mathbb{C}^B$ and $F_W^{\alpha,\beta,\gamma}\in\mathbb{C}^W$, respectively.

For every black vertex $(i,j,k)$, we have
\[f^{\alpha,\beta,\gamma}(i+1,j,k)+f^{\alpha,\beta,\gamma}(i,j+1,k)+f^{\alpha,\beta,\gamma}(i,j,k+1)=(\alpha+\beta+\gamma)\alpha^i\beta^j\gamma^k=(\alpha+\beta+\gamma)f_B^{\alpha,\beta,\gamma}(i,j,k),\]
and so by linearity, 
\[F^{\alpha,\beta,\gamma}(i+1,j,k)+F^{\alpha,\beta,\gamma}(i,j+1,k)+F^{\alpha,\beta,\gamma}(i,j,k+1)=(\alpha+\beta+\gamma)F_B^{\alpha,\beta,\gamma}(i,j,k).\]

For every $(i,j,k)\in L_1$, we have $j=k$, so from the definitions of $F^{\alpha,\beta,\gamma}$ and $f^{\alpha,\beta,\gamma}$,
\[F^{\alpha,\beta,\gamma}(i,j,k)=\alpha^i\beta^j\gamma^j-\beta^i\alpha^j\gamma^j-\alpha^i\gamma^j\beta^j-\gamma^i\beta^j\alpha^j+\beta^i\gamma^j\alpha^j+\gamma^i\alpha^j\beta^j=0.\]
Similarly,  $F^{\alpha,\beta,\gamma}$ vanishes on the boundary lines $L_2$ and $L_3$, where for $L_2$ we used that $\alpha^{3n+3}=\beta^{3n+3}=\gamma^{3n+3}$. More generally, $F^{\alpha,\beta,\gamma}$ vanishes whenever two of the coordinates differ by a multiple of $3n+3$.

\begin{claim}\label{claim:eigen}
For every $(i,j,k)\in B$, we have $(M_nF_W^{\alpha,\beta,\gamma})(i,j,k)=(\alpha+\beta+\gamma)F_B^{\alpha,\beta,\gamma}(i,j,k)$. Thus,
\[M_nF_W^{\alpha,\beta,\gamma}=(\alpha+\beta+\gamma)F_B^{\alpha,\beta,\gamma}.\]
\end{claim}
\begin{proof}[Proof of Claim~\ref{claim:eigen}]
First, let $(i,j,k)$ be a black vertex with three neighbours in $\mathcal{G}_n$. These three neighbours are $(i+1,j,k),(i,j+1,k),(i,j,k+1)\in W$, and so 
\begin{align*}
(M_nF_W^{\alpha,\beta,\gamma})(i,j,k)&=F_W^{\alpha,\beta,\gamma}(i+1,j,k)+F_W^{\alpha,\beta,\gamma}(i,j+1,k)+F_W^{\alpha,\beta,\gamma}(i,j,k+1)\\
&=(\alpha+\beta+\gamma)F_B^{\alpha,\beta,\gamma}(i,j,k).
\end{align*}

Now let $(i,j,k)$ be a black vertex such that $(i,j,k+1)$ lies on the boundary $L_1$. From above, $F^{\alpha,\beta,\gamma}$ vanishes at $(i,j,k+1)$, so
\begin{align*}
(M_nF_W^{\alpha,\beta,\gamma})(i,j,k)&=F_W^{\alpha,\beta,\gamma}(i+1,j,k)+F_W^{\alpha,\beta,\gamma}(i,j+1,k)\\
&=F^{\alpha,\beta,\gamma}(i+1,j,k)+F^{\alpha,\beta,\gamma}(i,j+1,k)+F^{\alpha,\beta,\gamma}(i,j,k+1)\\
&=(\alpha+\beta+\gamma)F_B^{\alpha,\beta,\gamma}(i,j,k).
\end{align*}
The cases when $(i,j,k)$ is a black vertex such that $(i+1,j,k)$ lies on the boundary $L_2$, or $(i,j+1,k)$ lies on the boundary $L_3$ follow similarly, using that $F^{\alpha,\beta,\gamma}$ vanishes on $L_2$ and $L_3$. This proves the claim.
\renewcommand{\qedsymbol}{$\boxdot$}
\end{proof}
\renewcommand{\qedsymbol}{$\square$}

\textbf{III. Remove redundancy and degeneracy}

Note that for any permutation $\sigma$ of $\{\alpha,\beta,\gamma\}$, $F^{\sigma(\alpha),\sigma(\beta),\sigma(\gamma)}=\textup{sgn}(\sigma)F^{\alpha,\beta,\gamma}$. If $(\alpha',\beta',\gamma')=\delta(\alpha,\beta,\gamma)$ for some $(3n+3)$-th root of unity $\delta$, then $F_B^{\alpha',\beta',\gamma'}=\delta F_B^{\alpha,\beta,\gamma}$, using that every $(i,j,k)\in B$ satisfies $i+j+k=1$, and similarly $F_W^{\alpha',\beta',\gamma'}=\delta^2 F_W^{\alpha,\beta,\gamma}$. Also, if $\alpha,\beta,\gamma$ are not pairwise distinct, then $F^{\alpha,\beta,\gamma}$ is the zero vector. Moreover, if $\{\alpha,\beta,\gamma\}=\{1,\omega,\omega^2\}$, then for every $(i,j,k)\in B$, using $i+j+k=1$, we have
\begin{align*}
F_B^{\alpha,\beta,\gamma}(i,j,k)&=\pm(\omega^{j+2k}-\omega^{i+2k}-\omega^{2j+k}-\omega^{2i+j}+\omega^{i+2j}+\omega^{2i+k})\\
&=\pm(\omega^{i+2j+2}-\omega^{2i+j+2}-\omega^{2i+j+1}-\omega^{2i+j}+\omega^{i+2j}+\omega^{i+2j+1})\\
&=\pm(\omega^{i+2j}-\omega^{2i+j})(1+\omega+\omega^2)=0,
\end{align*}
so $F_B^{\alpha,\beta,\gamma}$ is the zero vector. Similarly, using $i+j+k=2$, $F_W^{\alpha,\beta,\gamma}$ is the zero vector. Combined with the above, we get that if $\{\alpha,\beta,\gamma\}=\{\delta,\delta\omega,\delta\omega^2\}$ for some $(3n+3)$-th root of unity $\delta$, then $F^{\alpha,\beta,\gamma}_B$ and $F^{\alpha,\beta,\gamma}_W$ are both the zero vector.

With the previous paragraph in mind, consider the set $\Omega'$ of triples of pairwise distinct $\alpha,\beta,\gamma\in\mathbb{C}$ satisfying $\alpha^{3n+3}=\beta^{3n+3}=\gamma^{3n+3}=1$, and $\{\alpha,\beta,\gamma\}\not=\{\delta,\delta\omega,\delta\omega^2\}$ for any $(3n+3)$-th root of unity $\delta$. Define an equivalence relation $\sim$ on $\Omega'$, where $(\alpha,\beta,\gamma)\sim(\alpha',\beta',\gamma')$ if and only if there is a permutation $\sigma$ of $\{\alpha,\beta,\gamma\}$ and a $(3n+3)$-th root of unity $\delta$ such that $(\sigma(\alpha),\sigma(\beta),\sigma(\gamma))=\delta(\alpha',\beta',\gamma')$. Then, the number of equivalence classes is 
\[\frac{(3n+3)(3n+2)(3n+1)-6(n+1)}{6(3n+3)}=\frac{3n(n+1)}2.\]
Let $\Omega\subset \Omega'$ be a set containing exactly one representative from each equivalence class of $\Omega'$ with respect to $\sim$.

\medskip
\textbf{IV. Linear independence}
\begin{claim}\label{claim:indep}
For every $\ast\in\{B,W\}$ and any $(\alpha,\beta,\gamma),(\alpha',\beta',\gamma')\in \Omega$, let
\[\langle F^{\alpha,\beta,\gamma}_{\ast},F^{\alpha',\beta',\gamma'}_{\ast}\rangle=\sum_{(i,j,k)\in\ast}F^{\alpha,\beta,\gamma}_{\ast}(i,j,k)\overline{F^{\alpha',\beta',\gamma'}_{\ast}(i,j,k)}\]
Then,
\[\langle F^{\alpha,\beta,\gamma}_{\ast},F^{\alpha',\beta',\gamma'}_{\ast}\rangle=
\begin{cases}
(3n+3)^2 & \text{ if }(\alpha,\beta,\gamma)=(\alpha',\beta',\gamma'),\\
0 & \text{ if }(\alpha,\beta,\gamma)\not=(\alpha',\beta',\gamma').    
\end{cases}\]
\end{claim}
\begin{proof}[Proof of Claim~\ref{claim:indep}]
We only prove the case when $\ast=B$, as the $\ast=W$ case is essentially the same. 

The set of black vertices can be parametrised as
\[B=\{(i,j,k)\in\mathbb{Z}^3:i>j>k>i-3n-3,\;i+j+k=1\}.\]
By permuting the coordinates of each $(i,j,k)\in B$, we obtain five other sets as shown in Figure~\ref{fig:symmetry}. Let $B^+=\cup_{(i,j,k)\in B}\{(i,j,k),(i,k,j),(j,i,k),(j,k,i),(k,i,j),(k,j,i)\}$, and note that 
\[B^+=\{(i,j,k)\in\mathbb{Z}^3:i+j+k=1,\;i,j,k\text{ distinct},\;\max\{i,j,k\}-\min\{i,j,k\}<3n+3\}.\]
Directly from the definition, $F^{\alpha,\beta,\gamma}(i,k,j)=F^{\alpha,\gamma,\beta}(i,j,k)=-F^{\alpha,\beta,\gamma}(i,j,k)$. Thus, \[F^{\alpha,\beta,\gamma}(i,k,j)\overline{F^{\alpha',\beta',\gamma'}(i,k,j)}=F^{\alpha,\beta,\gamma}(i,j,k)\overline{F^{\alpha',\beta',\gamma'}(i,j,k)},\]
and this similarly holds for other permutations of $(i,j,k)$. Therefore, 
\[\langle F^{\alpha,\beta,\gamma}_B,F^{\alpha',\beta',\gamma'}_B\rangle=\sum_{(i,j,k)\in B}F^{\alpha,\beta,\gamma}_B(i,j,k)\overline{F^{\alpha',\beta',\gamma'}_B(i,j,k)}=\frac16\sum_{(i,j,k)\in B^+}F^{\alpha,\beta,\gamma}(i,j,k)\overline{F^{\alpha',\beta',\gamma'}(i,j,k)}.\]

\begin{figure}[h]
    \centering
    \input{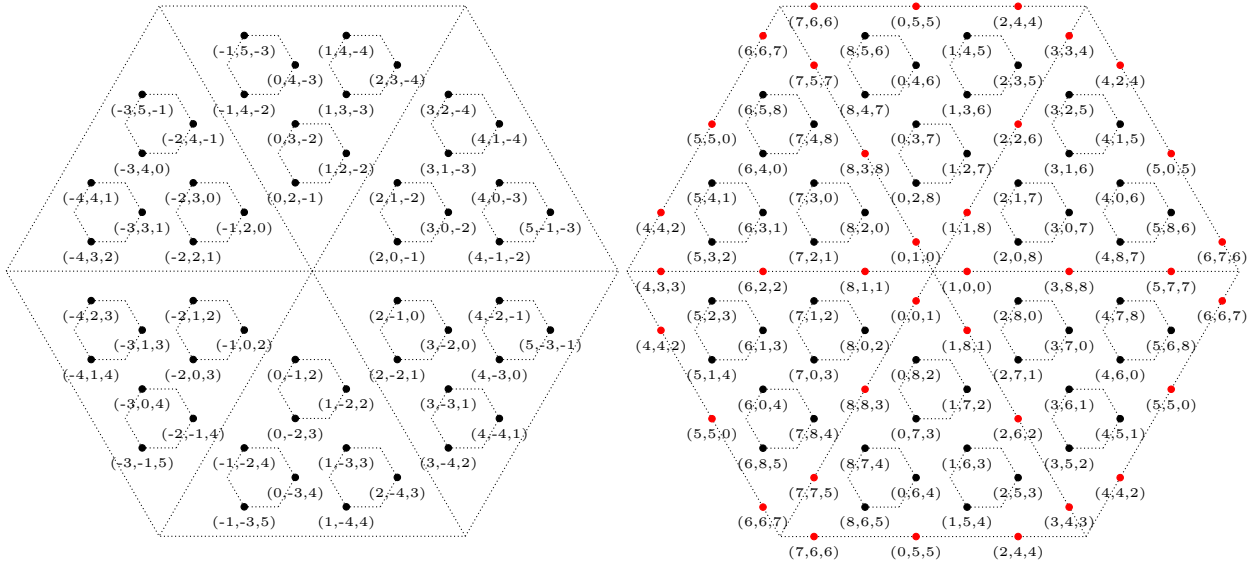}
    \caption{An illustration of the sets $B^+$ and $\overline{B^+}$ when $n=2$. On the left, the set $B^+$ obtained by permuting the coordinates of each point in $B$. On the right, the black points are now labelled with their new coordinates in $\overline{B^+}$. The red points are the rest of the points in $\overline{B}$, with repetition. They all lie on the boundaries and have repeated coordinates.}\label{fig:symmetry}
\end{figure}

Let $\overline{B^+}=\{(i,j,k)\in\mathbb{Z}^3:0\leq i,j,k\leq 3n+2,\;i,j,k\text{ distinct},\;i+j+k\equiv1\pmod{3n+3}\}$. Define $\Phi:B^+\to\overline{B^+}$ by $\Phi(i,j,k)=(i',j',k')$, where $i'=i+3n+3$ if $i<0$ and $i'=i$ otherwise, and similarly for $j'$ and $k'$. Since $\max\{i,j,k\}-\min\{i,j,k\}<3n+3$ and $i+j+k=1$, each of $i,j,k$ is at least $-3n-2$, and if two of $i',j',k'$ are equal then their counterparts in $i,j,k$ are equal as well. Thus, $\Phi$ does map into $\overline{B^+}$. We claim that $\Phi$ is a bijection. Indeed, first suppose there exist $(i_1,j_1,k_1)\not=(i_2,j_2,k_2)$, but $\Phi(i_1,j_1,k_1)=\Phi(i_2,j_2,k_2)$. Then, since $i_1+j_1+k_1=i_2+j_2+k_2=1$, and each corresponding coordinate pair can differ by at most $3n+3$, we can assume, without loss of generality, that $i_1=i_2+3n+3$, and $j_1=j_2-3n-3$. Then, $i_1-j_1=i_2-j_2+6n+6$, so one of $|i_1-j_1|$ and $|i_2-j_2|$ is at least $3n+3$, a contradiction. Thus, $\Phi$ is injective. Now let $(i',j',k')\in\overline{B^+}$, and assume without loss of generality that $i'<j'<k'$. If $i'+j'+k'=1$, then $\Phi(i',j',k')=(i',j',k')$. If $i'+j'+k'=3n+4$, then $(i',j',k'-3n-3)\in B^+$, and $\Phi(i',j',k'-3n-3)=(i',j',k')$. If $i'+j'+k'=6n+7$, then $(i',j'-3n-3,k'-3n-3)\in B^+$, and $\Phi(i',j'-3n-3,k'-3n-3)=(i',j',k')$. Thus, $\Phi$ is surjective and therefore bijective. 

Recall that $\alpha,\beta,\gamma,\alpha',\beta',\gamma'$ are all $(3n+3)$-th roots of unity, so $F^{\alpha,\beta,\gamma}(i,j,k)=F^{\alpha,\beta,\gamma}(\Phi(i,j,k))$ for all $(i,j,k)\in B^+$, and similarly for $F^{\alpha',\beta',\gamma'}$. Also recall that $F^{\alpha,\beta,\gamma}$ and $F^{\alpha',\beta',\gamma'}$ both vanish on all points $(i,j,k)$ that contain two coordinates that differ by a multiple of $3n+3$. Therefore, if $\overline{B}=\{(i,j,k)\in\mathbb{Z}^3:0\leq i,j,k\leq 3n+2,\;i+j+k\equiv1\pmod{3n+3}\}$, then
\begin{align*}
\langle F^{\alpha,\beta,\gamma}_B,F^{\alpha',\beta',\gamma'}_B\rangle&=\frac16\sum_{(i,j,k)\in B^+}F^{\alpha,\beta,\gamma}(i,j,k)\overline{F^{\alpha',\beta',\gamma'}(i,j,k)}\\
&=\frac16\sum_{(i,j,k)\in B^+}F^{\alpha,\beta,\gamma}(\Phi(i,j,k))\overline{F^{\alpha',\beta',\gamma'}(\Phi(i,j,k))}\\
&=\frac16\sum_{(i,j,k)\in\overline{B^+}}F^{\alpha,\beta,\gamma}(i,j,k)\overline{F^{\alpha',\beta',\gamma'}(i,j,k)}\\
&=\frac16\sum_{(i,j,k)\in\overline{B}}F^{\alpha,\beta,\gamma}(i,j,k)\overline{F^{\alpha',\beta',\gamma'}(i,j,k)}.
\end{align*}

To continue, we first consider the sum above with $f$ in place of $F$. For any $(3n+3)$-th roots of unity $x,y,z,x',y',z'$, we have
\begin{align*}
\sum_{(i,j,k)\in\overline{B}}f^{x,y,z}(i,j,k)\overline{f^{x',y',z'}(i,j,k)}&=\sum_{(i,j,k)\in\overline{B}}\left(\frac x{x'}\right)^i\left(\frac{y}{y'}\right)^j\left(\frac{z}{z'}\right)^k\\
&=\sum_{i=0}^{3n+2}\sum_{j=0}^{3n+2}\left(\frac x{x'}\right)^i\left(\frac{y}{y'}\right)^j\left(\frac{z}{z'}\right)^{1-i-j}\\
&=\frac z{z'}\left(\sum_{i=0}^{3n+2}\left(\frac{xz'}{x'z}\right)^{i}\right)\left(\sum_{j=0}^{3n+2}\left(\frac{yz'}{y'z}\right)^j\right).
\end{align*}
This is 0 unless $(x,y,z)$ is a scalar multiple of $(x',y',z')$, in which case this is $(3n+3)^2z/z'$.

Recall that $(\alpha,\beta,\gamma),(\alpha',\beta',\gamma')\in \Omega$, so if $(\alpha,\beta,\gamma)\not=(\alpha',\beta',\gamma')$, then from the definition of the equivalence relation $\sim$, $(\sigma(\alpha),\sigma(\beta),\sigma(\gamma))$ is not a scalar multiple of $(\alpha',\beta',\gamma')$ for any permutation $\sigma$ of $\{\alpha,\beta,\gamma\}$. Thus, by expanding both $F$ terms into $f$ terms,
\[\langle F^{\alpha,\beta,\gamma}_B,F^{\alpha',\beta',\gamma'}_B\rangle=\frac16\sum_{(i,j,k)\in\overline{B}}F^{\alpha,\beta,\gamma}(i,j,k)\overline{F^{\alpha',\beta',\gamma'}(i,j,k)}=0.\]
On the other hand, if $(\alpha,\beta,\gamma)=(\alpha',\beta',\gamma')$, then using that $\alpha,\beta,\gamma$ are pairwise distinct, and $\{\alpha,\beta,\gamma\}\not=\{\delta,\delta\omega,\delta\omega^2\}$ for any $(3n+3)$-th root of unity $\delta$, we have
\[\langle F^{\alpha,\beta,\gamma}_B,F^{\alpha,\beta,\gamma}_B\rangle=\frac16\sum_{(i,j,k)\in\overline{B}}F^{\alpha,\beta,\gamma}(i,j,k)\overline{F^{\alpha,\beta,\gamma}(i,j,k)}=\frac16\cdot6(3n+3)^2=(3n+3)^2.\]
This proves the claim.
\renewcommand{\qedsymbol}{$\boxdot$}
\end{proof}
\renewcommand{\qedsymbol}{$\square$}

For each $\ast\in\{B,W\}$, let $V_\ast=\{F^{\alpha,\beta,\gamma}_{\ast}:(\alpha,\beta,\gamma)\in \Omega\}$. Since $|\Omega|=3n(n+1)/2$ is equal to the dimension of $M_n$, by Claim~\ref{claim:indep} and standard linear algebra, $V_B$ is a basis of $\mathbb{C}^B$ and $V_W$ is a basis of $\mathbb{C}^W$. Then, Claim~\ref{claim:eigen} implies that when represented in the new bases $V_B$ and $V_W$, $M_n$ becomes diagonal with the set of diagonal entries being $\{\alpha+\beta+\gamma:(\alpha,\beta,\gamma)\in \Omega\}$. 

\medskip
\textbf{V. Determinant calculation}

Note that Claim~\ref{claim:indep} also implies that if $M_B$ is the matrix whose columns are vectors in $V_B$, and $M_W$ is the matrix whose columns are vectors in $V_W$, then $|\det(M_B)|=|\det(M_W)|=(3n+3)^{|\Omega|}$. Therefore,
\[|\det(M_n)|=\frac{|\det(M_B)|}{|\det(M_W)|}\prod_{(\alpha,\beta,\gamma)\in \Omega}|\alpha+\beta+\gamma|=\prod_{(\alpha,\beta,\gamma)\in \Omega}|\alpha+\beta+\gamma|.\]

Let $\Lambda=\{(1,\zeta^a,\zeta^b):1\leq a<b\leq 3n+2,(a,b)\not=(n+1,2n+2)\}$, and note that $|\Lambda|=(3n+2)(3n+1)/2-1=9n(n+1)/2=3|\Omega|$. Let $H$ be an auxiliary bipartite graph with bipartition classes $\Omega$ and $\Lambda$, such that $(\alpha,\beta,\gamma)\in \Omega$ and $(1,\zeta^a,\zeta^b)\in \Lambda$ are connected with an edge if and only if $(\alpha,\beta,\gamma)\sim(1,\zeta^a,\zeta^b)$. For every $(\alpha,\beta,\gamma)\in \Omega$, by permuting $\alpha,\beta,\gamma$, there exists $0\leq p<q<r\leq 3n+2$ such that $(\alpha,\beta,\gamma)\sim(\zeta^p,\zeta^q,\zeta^r)$. Let $u=q-p$, $v=r-q$, and $w=p-r+3n+3$. Then $u,v,w>0$, $u+v+w=3n+3$, and $(u,v,w)\not=(n+1,n+1,n+1)$. Furthermore, by dividing $(\zeta^p,\zeta^q,\zeta^r)$ by each of $\zeta^p$, $\zeta^q$, $\zeta^r$, and then reordering, we get $(\alpha,\beta,\gamma)\sim(1,\zeta^u,\zeta^{u+v})\sim(1,\zeta^v,\zeta^{v+w})\sim(1,\zeta^w,\zeta^{w+u})$. Moreover, the last three triples are pairwise distinct elements of $\Lambda$, so every vertex in $\Omega$ has degree at least 3 in $H$. On the other hand, since $\Omega$ contains at most one element from each equivalence class, every vertex in $\Lambda$ has degree at most 1 in $H$. Thus, every vertex in $\Omega$ has degree exactly 3, and every vertex in $\Lambda$ has degree exactly 1. Therefore,
\[T_n=|\det(M_n)|=\prod_{(\alpha,\beta,\gamma)\in \Omega}|\alpha+\beta+\gamma|=\prod_{(1,\zeta^a,\zeta^b)\in \Lambda}|1+\zeta^a+\zeta^b|^{1/3}=\prod_{\substack{1\leq a<b\leq 3n+2\\(a,b)\not=(n+1,2n+2)}}|1+\zeta^a+\zeta^b|^{1/3},\]
which completes the proof.
\end{proof}

\section{Prime factors}\label{sec:factor}
In this section, we prove Theorem~\ref{thm:factor}, which gives both an asymptotic growth rate for $T_n$ and explanations of why $T_n$ has many prime factors. 

We begin by recalling the following facts in Galois Theory. 
\begin{lemma}\label{lemma:rootgalois}
Let $N$ be a positive integer, and let $\xi=e^{2\pi i/N}$. Then, $\mathbb{Q}(\xi)/\mathbb{Q}$ is a Galois extension, and the Galois group $\textup{Gal}(\mathbb{Q}(\xi)/\mathbb{Q})$ is isomorphic to $\mathbb{Z}_N^\times$ via the map $\sigma_t\mapsto t$ for every $t\in\mathbb{Z}_N^\times$, where $\sigma_t$ is the automorphism of $\mathbb{Q}(\xi)$ mapping $\xi$ to $\xi^t$. If $z\in\mathbb{Q}(\xi)$ is fixed by $\sigma_t$ for all $t\in\mathbb{Z}_N^\times$, then $z\in\mathbb{Q}$.
\end{lemma}

\begin{lemma}\label{lemma:minpoly}
Let $K/\mathbb{Q}$ be a finite Galois extension, and let $z\in K$ be an algebraic integer. Then, the set of distinct conjugates of $z$ over $\mathbb{Q}$ is $C_z=\{\sigma(z):\sigma\in\textup{Gal}(K/\mathbb{Q})\}$. The minimal polynomial of $z$ over $\mathbb{Q}$ is $\prod_{y\in C_z}(x-y)$, and is monic with integer coefficients.
\end{lemma}

For any non-zero polynomial $p(x_1,\ldots,x_n)\in\mathbb{C}[x_1,\ldots,x_n]$, Mahler~\cite{Mah} defined what is now known as the Mahler measure of $p$ as
\[M(p)=\exp\left(\int_0^1\cdots\int_0^1\log|p(e^{2\pi i t_1},\ldots,e^{2\pi i t_n})|\text{d}t_1\cdots\text{d}t_n\right).\]
This can be viewed as a logarithmic average of $p$ on the unit torus. To prove~\ref{factor3}, we need the following result of Smyth~{\cite{S}} on the Mahler measure of the bivariate polynomial $1+x+y$.
\begin{lemma}[{\cite[Example 5]{S}}]\label{lemma:smyth}
Let $\chi_{-3}:\mathbb{Z}\to\{0,\pm1\}$ be the primitive quadratic character modulo 3 given by $\chi_{-3}(n)=0$ if $3\mid n$, $\chi_{-3}(n)=1$ if $n\equiv 1\pmod 3$, and $\chi_{-3}(n)=-1$ otherwise. Let $L(2,\chi_{-3})=\sum_{n=1}^\infty\chi_{-3}(n)n^{-2}$ be the Dirichlet $L$-function for $\chi_{-3}$ evaluated at 2.
Then, 
\[M(1+x+y)=\exp\left(\frac{3\sqrt3}{4\pi}L(2,\chi_{-3})\right).\]
\end{lemma}

Equipped with the results above, we can now prove Theorem~\ref{thm:factor}.

\begin{proof}[Proof of Theorem~\ref{thm:factor}]
For every positive integer $n$, let $I_n=\{(a,b)\in\mathbb{Z}^2:1\leq a<b\leq 3n+2,\;(a,b)\not=(n+1,2n+2)\}$. For every $(a,b)\in I_n$, let $z_{(a,b)}=-\zeta^a-\zeta^b$. Since $\zeta$ is an algebraic integer, and the set of algebraic integers is closed under addition and multiplication, $z_{(a,b)}$ is an algebraic integer for every $(a,b)\in I_n$, so its minimal polynomial over $\mathbb{Q}$ is in $\mathbb{Z}[x]$. Define a monic polynomial $P_n(x)$ with $|P_n(1)|=T_n^3$ by
\[P_n(x)=\prod_{(a,b)\in I_n}(x+\zeta^a+\zeta^b)=\prod_{(a,b)\in I_n}(x-z_{(a,b)}).\]

\ref{factor1}: By Lemma~\ref{lemma:rootgalois}, the elements in $\textup{Gal}(\mathbb{Q}(\zeta)/\mathbb{Q})$ are exactly the automorphisms $\sigma_t$ for every $t\in\mathbb{Z}_{3n+3}^\times$, which sends $\zeta$ to $\zeta^t$. Let $t\in\mathbb{Z}_{3n+3}^\times$. For every $(a,b)\in I_n$, there exist $1\leq a',b'\leq 3n+2$ satisfying $a'\equiv ta\pmod{3n+3}$ and $b'\equiv tb\pmod{3n+3}$. Also, multiplication by any $t\in\mathbb{Z}_{3n+3}^\times$ fixes the set $\{n+1,2n+2\}$, so $\{a',b'\}\not=\{n+1,2n+2\}$, as otherwise $\{a,b\}=\{n+1,2n+2\}$. Define $\psi_t(a,b)=(a',b')$ if $a'<b'$, and $\psi_t(a,b)=(b',a')$ otherwise. Then, $\psi_t$ is a map from $I_n$ to $I_n$ with inverse $\psi_{t^{-1}}$, and hence a bijection. Thus, for every $t\in\mathbb{Z}_{3n+3}^\times$, applying $\sigma_t$ to the coefficients of $P_n(x)$ gives
\[\prod_{(a,b)\in I_n}(x+\zeta^{at}+\zeta^{bt})=\prod_{(a,b)\in I_n}(x-z_{\psi_t(a,b)})=\prod_{(a,b)\in I_n}(x-z_{(a,b)})=P_n(x).\]
Since the coefficients of $P_n(x)$ are fixed by every element of $\textup{Gal}(\mathbb{Q}(\zeta)/\mathbb{Q})$, $P_n(x)\in\mathbb{Q}[x]$ by Lemma~\ref{lemma:rootgalois}. Furthermore, since each $z_{(a,b)}$ is an algebraic integer, each coefficient of $P_n(x)$ is an algebraic integer too. Thus, $P_n(x)\in\mathbb{Z}[x]$. It follows that $P_n(1)$ is an integer, so $P_n(1)=\pm T_n^3$.

Let $p$ be any monic irreducible factor of $P_n(x)$ in $\mathbb{Q}[x]$, and let $z_{(a,b)}$ be a root of $p$. Then, $p$ is the minimal polynomial of $z_{(a,b)}$, so $p\in\mathbb{Z}[x]$. Furthermore, by Lemma~\ref{lemma:minpoly}, the degree of $p$ is at most $|\mathbb{Z}_{3n+3}^\times|=\varphi(3n+3)\leq 2n+2$. Thus, $P_n(x)$ can be factored into polynomials in $\mathbb{Z}[x]$, each of degree at most $2n+2$.


\ref{factor2}: Let $d=(n+1)/(m+1)$. Let $\eta=e^{2\pi i/(3m+3)}$, so $\eta=\zeta^d$. Then, the map $g:I_m\to I_n$ given by $g(a',b')=(a'd,b'd)$ is injective, so
\[P_m(x)=\prod_{(a',b')\in I_m}(x+\eta^{a'}+\eta^{b'})=\prod_{(a',b')\in I_m}(x+\zeta^{a'd}+\zeta^{b'd})=\prod_{(a,b)\in g(I_m)}(x+\zeta^{a}+\zeta^{b})\]
divides $P_n(x)$ in $\mathbb{C}[x]$. But both $P_m(x)$ and $P_n(x)$ are monic polynomials in $\mathbb{Z}[x]$, so $P_m(x)$ divides $P_n(x)$ in $\mathbb{Z}[x]$. It follows that $P_m(1)$ divides $P_n(1)$, and so $T_m=|P_m(1)|^{1/3}$ divides $T_n=|P_n(1)|^{1/3}$.

\ref{factor3}: Taking logarithms on both sides of the formula for $T_n$ in Theorem~\ref{thm:main}, we get
\begin{align*}
\log(T_n)&=\frac13\sum_{\substack{1\leq a<b\leq 3n+2\\(a,b)\not=(n+1,2n+2)}}\log|1+\zeta^a+\zeta^b|=\frac16\sum_{\substack{1\leq a,b\leq 3n+2\\a\not=b,\;\{a,b\}\not=\{n+1,2n+2\}}}\log|1+\zeta^a+\zeta^b|\\
&=\frac16\sum_{\substack{0\leq a,b\leq 3n+2\\\{a,b\}\not=\{n+1,2n+2\}}}\log|1+\zeta^a+\zeta^b|-\frac13\sum_{a=0}^{3n+2}\log|2+\zeta^a|-\frac16\sum_{a=0}^{3n+2}\log|1+2\zeta^a|+\frac13\log 3\\
&=\frac16\sum_{\substack{0\leq a,b\leq 3n+2\\\{a,b\}\not=\{n+1,2n+2\}}}\log|1+\zeta^a+\zeta^b|-\frac13\log|(-2)^{3n+3}-1|-\frac16\log|(-1)^{3n+3}-2^{3n+3}|+\frac13\log 3\\
&=\frac16\left(\sum_{\substack{0\leq a,b\leq 3n+2\\\{a,b\}\not=\{n+1,2n+2\}}}\log|1+\zeta^a+\zeta^b|\right)\pm O(n).
\end{align*}
Then, dividing by $(3n+3)^2$ and interpreting the sum as a Riemann sum, we have from Lemma~\ref{lemma:smyth} that
\begin{align*}
\lim_{n\to\infty}\frac1{(3n+3)^2}\log(T_n)&=\lim_{n\to\infty}\frac1{6\cdot(3n+3)^2}\sum_{\substack{0\leq a,b\leq 3n+2\\\{a,b\}\not=\{n+1,2n+2\}}}\log|1+\zeta^a+\zeta^b|\\
&=\frac16\int_{0}^1\int_{0}^1\log|1+e^{2\pi it_1}+e^{2\pi i t_2}|\text{d}t_1\text{d}t_2\\
&=\frac16\log(M(1+x+y))=\frac{\sqrt 3}{8\pi}L(2,\chi_{-3}),
\end{align*}
where the second equality follows since around the singularities $(1/3,2/3)$ and $(2/3,1/3)$, $\log|1+x+y|$ tend to $-\infty$ logarithmically slowly, so the contribution of points near them are negligible. 
Therefore,
\[T_n=\exp\left(\frac{9\sqrt 3}{8\pi}L(2,\chi_{-3})n^2+o(n^2)\right)\approx(1.6235)^{n^2}.\]

\ref{factor4}: Factor $P_n(x)$ into monic irreducible polynomials in $\mathbb{Q}[x]$, and let $\mathcal{P}_n$ be the multiset of such irreducible factors. For each $p\in\mathcal{P}_n$, from the proof of~\ref{factor1}, $p$ is the minimal polynomial of $z_{(a,b)}$ for some $(a,b)\in I_n$, so it is in $\mathbb{Z}[x]$. Moreover, by Lemma~\ref{lemma:rootgalois} and Lemma~\ref{lemma:minpoly}, $|p(1)|$ is equal to the product of $|1-z|$ over the at most $2n+2$ conjugates $z$ of $z_{(a,b)}$, each of the form $z=-\zeta^{a'}-\zeta^{b'}$ for some $(a',b')\in I_n$. Since $|1-z|\leq 3$ for each such conjugate, we have $|p(1)|\leq 3^{2n+2}$. But $\prod_{p\in \mathcal{P}_n}|p(1)|=|P_n(1)|=T_n^3\geq 4^{n^2}$ for large $n$, so at least $\log_{3^{2n+2}}(4^{n^2})\geq 3n/5$ factors $p\in \mathcal{P}_n$ satisfy $|p(1)|>1$. Thus, $|P_n(1)|$ has at least $3n/5$ prime factors, counting multiplicity, so $T_n=|P_n(1)|^{1/3}$ has at least $n/5$ prime factors, counting multiplicity.
\end{proof} 

\section*{Acknowledgement}
The author acknowledges assistance from ChatGPT, which conjectured an equivalent version of the formula for $T_n$ and suggested the useful references~\cite{KPW,S}. The proofs in this paper, as well as the writing and presentation, are the work of the author, who takes full responsibility for the content. 

After putting this paper on arXiv, the author was notified by Christian Krattenthaler that independently, during a series of email exchanges in 2024, an equivalent formula for $T_n$ was found from experimentation by Noam Elkies, then a very similar proof of Theorem~\ref{thm:main} was communicated by Greg Kuperberg.

\bibliographystyle{plain}
\bibliography{bibliography}
\end{document}